\documentclass[oneside,english,reqno]{amsart}
\usepackage[T1]{fontenc}
\usepackage[latin9]{inputenc}
\usepackage{babel}
\usepackage{mathrsfs}
\usepackage{mathtools}
\usepackage{enumitem}
\usepackage{amstext}
\usepackage{amsthm}
\usepackage{amssymb}
\usepackage{graphicx}
\usepackage[unicode=true,pdfusetitle,
 bookmarks=true,bookmarksnumbered=false,bookmarksopen=false,
 breaklinks=false,pdfborder={0 0 1},backref=false,colorlinks=false]
 {hyperref}
\hypersetup{
 colorlinks=true,citecolor=blue,linkcolor=blue,linktocpage=true}

\makeatletter
\numberwithin{equation}{section}
\numberwithin{figure}{section}
\numberwithin{table}{section}
\theoremstyle{plain}
\newtheorem{thm}{\protect\theoremname}[section]
  \theoremstyle{definition}
  \newtheorem{defn}[thm]{\protect\definitionname}
  \theoremstyle{plain}
  \newtheorem{lem}[thm]{\protect\lemmaname}
  \theoremstyle{plain}
  \newtheorem{cor}[thm]{\protect\corollaryname}
  \theoremstyle{remark}
  \newtheorem{rem}[thm]{\protect\remarkname}
  \theoremstyle{plain}
  \newtheorem{prop}[thm]{\protect\propositionname}
  \theoremstyle{plain}
  \newtheorem*{lem*}{\protect\lemmaname}
  \theoremstyle{remark}
  \newtheorem*{note*}{\protect\notename}
  \theoremstyle{remark}
  \newtheorem*{acknowledgement*}{\protect\acknowledgementname}

\@ifundefined{date}{}{\date{}}
\allowdisplaybreaks
\usepackage{needspace}
\usepackage{refstyle}
\usepackage{enumitem}
\usepackage[skins,breakable]{tcolorbox}

\newref{lem}{refcmd={Lemma \ref{#1}}}
\newref{thm}{refcmd={Theorem \ref{#1}}}
\newref{cor}{refcmd={Corollary \ref{#1}}}
\newref{sec}{refcmd={Section \ref{#1}}}
\newref{sub}{refcmd={Section \ref{#1}}}
\newref{subsec}{refcmd={Section \ref{#1}}}
\newref{chap}{refcmd={Chapter \ref{#1}}}
\newref{prop}{refcmd={Proposition \ref{#1}}}
\newref{exa}{refcmd={Example \ref{#1}}}
\newref{tab}{refcmd={Table \ref{#1}}}
\newref{rem}{refcmd={Remark \ref{#1}}}
\newref{def}{refcmd={Definition \ref{#1}}}
\newref{fig}{refcmd={Figure \ref{#1}}}

\setlist[enumerate]{itemsep=5pt,topsep=3pt}
\setlist[enumerate,1]{label=(\roman*),ref=\roman*}
\setlist[enumerate,2]{label=(\alph*),ref=\theenumi \alph*}

\AtBeginDocument{
  
}

\makeatother

  \providecommand{\acknowledgementname}{Acknowledgement}
  \providecommand{\corollaryname}{Corollary}
  \providecommand{\definitionname}{Definition}
  \providecommand{\lemmaname}{Lemma}
  \providecommand{\notename}{Note}
  \providecommand{\propositionname}{Proposition}
  \providecommand{\remarkname}{Remark}
\providecommand{\theoremname}{Theorem}

\begin{document}

\title{Positive definite (p.d.) functions vs p.d. distributions}

\author{Palle Jorgensen}

\address{(Palle E.T. Jorgensen) Department of Mathematics, The University
of Iowa, Iowa City, IA 52242-1419, U.S.A. }

\email{palle-jorgensen@uiowa.edu}

\urladdr{http://www.math.uiowa.edu/\textasciitilde{}jorgen/}

\author{Feng Tian}

\address{(Feng Tian) Department of Mathematics, Hampton University, Hampton,
VA 23668, U.S.A.}

\email{feng.tian@hamptonu.edu}


\begin{abstract}
We give explicit transforms for Hilbert spaces associated with positive
definite functions on $\mathbb{R}$, and positive definite tempered
distributions, incl., generalizations to non-abelian locally compact
groups. Applications to the theory of extensions of p.d. functions/distributions
are included. We obtain explicit representation formulas for positive
definite tempered distributions in the sense of L. Schwartz, and we
give applications to Dirac combs and to diffraction. As further applications,
we give parallels between Bochner's theorem (for continuous p.d. functions)
on the one hand, and the generalization to Bochner/Schwartz representations
for positive definite tempered distributions on the other; in the
latter case, via tempered positive measures. Via our transforms, we
make precise the respective reproducing kernel Hilbert spaces (RKHSs),
that of N. Aronszajn and that of L. Schwartz. Further applications
are given to stationary-increment Gaussian processes. 

\tableofcontents{}
\end{abstract}

\maketitle
The study of positive definite (p.d.) functions and p.d. kernels is
motivated by diverse themes in analysis and operator theory, in white
noise analysis, applications of reproducing kernel (RKHS) theory,
extensions by Laurent Schwartz, and in reflection positivity from
quantum physics. Below we make more precise some parallels between,
on the one hand, the standard case from Case 1, of continuous positive
definite functions $f$ on $\mathbb{R}$, the setting of Bochner's
theorem, including generalizations to non-abelian locally compact
groups. We shall also discuss the theory of extensions of p.d. functions.
In part two of the paper we obtain representation formulas for positive
definite tempered distributions in the sense of L. Schwartz \cite{MR0185423,MR0179587}.
The parallels between Bochner's theorem (for continuous p.d. functions),
and the generalization to Bochner/Schwartz representations for positive
definite tempered distributions will be made clear. In the first case,
we have the Bochner representation via finite positive measures $\mu$;
and in the second case, instead via tempered positive measures. This
parallel also helps make precise the respective reproducing kernel
Hilbert spaces (RKHSs). This further leads to a more unified approach
to the treatment of the stationary-increment Gaussian processes \cite{MR2793121,MR2966130,MR3402823}.
A key argument will rely on the existence of a unitary representation
$U$ of $\left(\mathbb{R},+\right)$, acting on the particular RKHS
under discussion. In fact, the same idea (with suitable modifications)
will also work in the wider context of locally compact groups. In
the abelian case, we shall make use of the Stone representation for
$U$ in the form of orthogonal projection valued measures; and in
more general settings, the Stone-Naimark-Ambrose-Godement (SNAG) representation
\cite{MR1503079}. 

\section{Preliminaries}

In our theorems and proofs, we shall make use the particular reproducing
kernel Hilbert spaces (RKHSs) which allow us to give explicit formulas
for our solutions. The general framework of RKHSs were pioneered by
Aronszajn in the 1950s \cite{MR0051437}; and subsequently they have
been used in a host of applications; e.g., \cite{MR2327597,MR2558684}.

\subsubsection{The RKHS $\mathscr{H}_{f}$}

For simplicity we focus on the case $G=\mathbb{R}$. 
\begin{defn}
\label{def:pdf}Let $\Omega$ be an open domain in $\mathbb{R}$.
A function $f:\Omega-\Omega\rightarrow\mathbb{C}$ is \emph{positive
definite } if 
\begin{equation}
\sum\nolimits _{i}\sum\nolimits _{j}c_{i}\overline{c}_{j}f\left(x_{i}-x_{j}\right)\geq0\label{eq:def-pd}
\end{equation}
for all finite sums with $c_{i}\in\mathbb{C}$, and all $x_{i}\in\Omega$.
We assume that all the p.d. functions are continuous and bounded.
\end{defn}
\begin{lem}[Two equivalent conditions for p.d.]
If $f$ is given continuous on $\mathbb{R}$, we have the following
two equivalent conditions for the positive definite property:
\begin{enumerate}
\item $\forall n\in\mathbb{N}$, $\forall\left\{ x_{i}\right\} _{1}^{n}$,
$\forall\left\{ c_{i}\right\} _{1}^{n}$, $x\in\mathbb{R}$, $c_{i}\in\mathbb{C}$,
\[
\sum_{i}\sum_{j}c_{i}\overline{c}_{j}f\left(x_{i}-x_{j}\right)\geq0;
\]
\item $\forall\varphi\in C_{c}\left(\mathbb{R}\right)$, we have: 
\[
\int_{\mathbb{R}}\int_{\mathbb{R}}\varphi\left(x\right)\overline{\varphi\left(y\right)}f\left(x-y\right)dxdy\geq0.
\]
\end{enumerate}
\end{lem}
\begin{proof}
Use Riemann integral approximation, and note that $f\left(\cdot-x\right)\in\mathscr{H}_{f}$
and $\varphi\ast f\in\mathscr{H}_{f}$. (See details below.)
\end{proof}

Consider a continuous positive definite function so $f$ is defined
on $\Omega-\Omega$. Set 
\begin{equation}
f_{y}\left(x\right):=f\left(x-y\right),\;\forall x,y\in\Omega.\label{eq:H0}
\end{equation}
Let $\mathscr{H}_{f}$ be the \emph{reproducing kernel Hilbert space
}(RKHS), which is the completion of 
\begin{equation}
\left\{ \sum\nolimits _{\text{finite}}c_{j}f_{x_{j}}\mid x_{j}\in\Omega,\:c_{j}\in\mathbb{C}\right\} \label{eq:H1}
\end{equation}
with respect to the inner product
\begin{equation}
\left\langle \sum\nolimits _{i}c_{i}f_{x_{i}},\sum\nolimits _{j}d_{j}f_{y_{j}}\right\rangle _{\mathscr{H}_{f}}:=\sum\nolimits _{i}\sum\nolimits _{j}c_{i}\overline{d}_{j}f\left(x_{i}-y_{j}\right);\label{eq:ip-discrete}
\end{equation}
modulo the subspace of functions of $\left\Vert \cdot\right\Vert _{\mathscr{H}_{f}}$-norm
zero. 

Below, we introduce an equivalent characterization of the RKHS $\mathscr{H}_{f}$,
which we will be working with in the rest of the paper. 
\begin{lem}
\label{lem:dense}Fix $\Omega=(0,\alpha)$. Let $\varphi_{n,x}\left(t\right)=n\varphi\left(n\left(t-x\right)\right)$,
for all $t\in\Omega$; where $\varphi$ satisfies 

\begin{enumerate}
\item $\mathrm{supp}\left(\varphi\right)\subset\left(-\alpha,\alpha\right)$;
\item $\varphi\in C_{c}^{\infty}$, $\varphi\geq0$;
\item $\int\varphi\left(t\right)dt=1$. Note that $\lim_{n\rightarrow\infty}\varphi_{n,x}=\delta_{x}$,
the Dirac measure at $x$.
\end{enumerate}
\end{lem}

\begin{lem}
\label{lem:RKHS-def-by-integral}The RKHS, $\mathscr{H}_{f}$, is
the Hilbert completion of the functions 
\begin{equation}
f_{\varphi}\left(x\right)=\int_{\Omega}\varphi\left(y\right)f\left(x-y\right)dy,\;\forall\varphi\in C_{c}^{\infty}\left(\Omega\right),x\in\Omega\label{eq:H2}
\end{equation}
with respect to the inner product
\begin{equation}
\left\langle f_{\varphi},f_{\psi}\right\rangle _{\mathscr{H}_{f}}=\int_{\Omega}\int_{\Omega}\varphi\left(x\right)\overline{\psi\left(y\right)}f\left(x-y\right)dxdy,\;\forall\varphi,\psi\in C_{c}^{\infty}\left(\Omega\right).\label{eq:hi2}
\end{equation}
In particular, 
\begin{equation}
\left\Vert f_{\varphi}\right\Vert _{\mathscr{H}_{f}}^{2}=\int_{\Omega}\int_{\Omega}\varphi\left(x\right)\overline{\varphi\left(y\right)}f\left(x-y\right)dxdy,\;\forall\varphi\in C_{c}^{\infty}\left(\Omega\right)\label{eq:hn2}
\end{equation}
and 
\begin{equation}
\left\langle f_{\varphi},f_{\psi}\right\rangle _{\mathscr{H}_{f}}=\int_{\Omega}f_{\varphi}\left(x\right)\overline{\psi\left(x\right)}dx,\;\forall\phi,\psi\in C_{c}^{\infty}(\Omega).
\end{equation}
\end{lem}
\begin{proof}
Indeed, by Lemma \ref{lem:RKHS-def-by-integral}, we have 
\begin{equation}
\left\Vert f_{\varphi_{n,x}}-f\left(\cdot-x\right)\right\Vert _{\mathscr{H}_{f}}\rightarrow0,\;\mbox{as }n\rightarrow\infty.\label{eq:approx}
\end{equation}
Hence $\left\{ f_{\varphi}\right\} _{\varphi\in C_{c}^{\infty}\left(\Omega\right)}$
spans a dense subspace in $\mathscr{H}_{f}$. 

For more details, see \cite{MR863534,MR874059}.

\end{proof}

These two conditions (\ref{eq:bdd})($\Leftrightarrow$(\ref{eq:bdd2}))
below will be used to characterize elements in the Hilbert space $\mathscr{H}_{f}$.
\begin{thm}
\label{thm:HF}A continuous function $\xi:\Omega\rightarrow\mathbb{C}$
is in $\mathscr{H}_{f}$ if and only if there exists $A_{0}>0$, such
that
\begin{equation}
\sum\nolimits _{i}\sum\nolimits _{j}c_{i}\overline{c}_{j}\xi\left(x_{i}\right)\overline{\xi\left(x_{j}\right)}\leq A_{0}\sum\nolimits _{i}\sum\nolimits _{j}c_{i}\overline{c}_{j}f\left(x_{i}-x_{j}\right)\label{eq:bdd}
\end{equation}
for all finite system $\left\{ c_{i}\right\} \subset\mathbb{C}$ and
$\left\{ x_{i}\right\} \subset\Omega$.

Equivalently, for all $\psi\in C_{c}^{\infty}\left(\Omega\right)$,
\begin{eqnarray}
\left|\int_{\Omega}\psi\left(y\right)\overline{\xi\left(y\right)}dy\right|^{2} & \leq & A_{0}\int_{\Omega}\int_{\Omega}\psi\left(x\right)\overline{\psi\left(y\right)}f\left(x-y\right)dxdy\label{eq:bdd2}
\end{eqnarray}
Note that, if $\xi\in\mathscr{H}_{f}$, then the LHS of (\ref{eq:bdd2})
is $\vert\left\langle f_{\psi},\xi\right\rangle _{\mathscr{H}_{f}}\vert^{2}$.
Indeed,
\begin{eqnarray*}
\left|\left\langle \xi,f_{\psi}\right\rangle _{\mathscr{H}_{f}}\right|^{2} & = & \left|\left\langle \xi,\int_{\Omega}\psi\left(y\right)f_{y}\:dy\right\rangle _{\mathscr{H}_{f}}\right|^{2}\\
 & = & \left|\int_{\Omega}\overline{\psi\left(y\right)}\left\langle \xi,f_{y}\right\rangle _{\mathscr{H}_{f}}dy\right|^{2}\\
 & = & \left|\int_{\Omega}\overline{\psi\left(y\right)}\xi\left(y\right)dy\right|^{2}.\;(\text{by the reproducing property})
\end{eqnarray*}
\end{thm}

\section{\label{sec:para}The parallels of p.d. functions vs distributions}

In this section we prove the following theorem:
\begin{thm}
\label{thm:p1}~
\begin{enumerate}
\item[(a)]  Let $f$ be a continuous positive definite (p.d.) function on $\mathbb{R}$
(a p.d. tempered distribution \cite{MR0185423,MR0179587}); then there
is a unique finite positive Borel measure $\mu$ on $\mathbb{R}$
(resp., a unique tempered measure on $\mathbb{R}$) such that $f=\widehat{\mu}$.
\item[(b)]  Given $f$ as above, let $\mathscr{H}_{f}$ denote the corresponding
kernel Hilbert space, i.e., the Hilbert completion of $\left\{ \varphi\ast f\right\} _{\varphi\in C_{c}\left(\Omega\right)}$
(resp. $\varphi\in\mathcal{S}$) w.r.t 
\[
\left\Vert \varphi\ast f\right\Vert _{\mathscr{H}_{f}}^{2}=\int_{\mathbb{R}}\int_{\mathbb{R}}\varphi\left(x\right)\overline{\varphi\left(y\right)}f\left(x-y\right)dxdy
\]
resp., $\left\langle f\left(x-y\right),\varphi\ast\overline{\varphi}\right\rangle $;
action in the sense of distributions. Then there is a unique isometric
transform
\begin{gather*}
\mathscr{H}_{f}\xrightarrow{\;T_{f}\;}L^{2}\left(\mathbb{R},\mathscr{B},\mu\right),\quad T_{f}\left(\varphi\ast f\right)=\widehat{\varphi},\;\text{i.e.,}\\
\left\Vert \varphi\ast f\right\Vert _{\mathscr{H}_{f}}^{2}=\int_{\mathbb{R}}\left|\widehat{\varphi}\right|^{2}d\mu=\left\Vert T_{f}\varphi\right\Vert _{L^{2}\left(\mu\right)}^{2}.
\end{gather*}
\item[(c)]  If $\mu$ is tempered, e.g., if $\int_{\mathbb{R}}\frac{d\mu\left(\lambda\right)}{1+\lambda^{2}}<\infty$,
then 
\[
\left\Vert \varphi\ast f\right\Vert _{\mathscr{H}_{f}}^{2}=\int\left(\left|\widehat{\varphi}\right|^{2}+\big|\widehat{\left(D_{x}\varphi\right)}\big|^{2}\right)\frac{d\mu\left(\lambda\right)}{1+\lambda^{2}};
\]
where $D_{x}\varphi=\frac{d\varphi}{dx}$, and where ``$\widehat{\,\cdot\,}$''
denotes the standard Fourier transform on $\mathbb{R}$. 
\end{enumerate}
\end{thm}
\begin{proof}
The proof will be given below. It will be divided up in a sequence
of Lemmas, which by their own merit might be of independent interest. 
\end{proof}
\begin{cor}
Let a function $f$ (or a tempered distribution) be given on a finite
open interval in $\mathbb{R}$, but assumed positive definite there;
then it automatically has a positive definite extension to $\mathbb{R}$
(in the same category), and so the conclusion of Theorem \ref{thm:p1}
still applies to $f$, or referring to the corresponding p.d. extension.
\end{cor}
\begin{proof}
The result uses a main theorem in \cite{MR3559001}, as well as Lemmas
\ref{lem:dense}-\ref{lem:RKHS-def-by-integral} above.
\end{proof}
\begin{rem}
In the case of tempered distributions, the inner product in the RKHS
$\mathscr{H}_{f}$ is as follows: First, given $f$ and $\varphi\in\mathcal{S}$,
$f$ a p.d. tempered distribution; then the convolution $f\ast\varphi$
is a well defined tempered distribution, and so, for $\psi\in\mathcal{S}$,
the expression $\left\langle f\ast\varphi,\overline{\psi}\right\rangle $
denotes the distribution $f\ast\varphi$ applied to $\overline{\psi}\in\mathcal{S}$.
Hence, the $\mathscr{H}_{f}$-inner product is: 
\[
\left\langle f\ast\varphi,f\ast\psi\right\rangle _{\mathscr{H}_{f}}:=\left\langle f\ast\varphi,\overline{\psi}\right\rangle ,
\]
with this interpretation of action of the distribution $f\ast\varphi$
and the test function $\overline{\psi}$. The p.d. property for $f$
amounts to 
\[
\left\langle f\ast\varphi,\overline{\varphi}\right\rangle \geq0,\;\forall\varphi\in\mathcal{S}.
\]
\end{rem}
\begin{rem}
The conclusions in the statement of Theorem \ref{thm:p1} hold also
when $\mathbb{R}$ is replaced with an arbitrary locally compact Abelian
group $G$. The modification are as follows: Now $f$ will instead
be a given continuous p.d. function (or a tempered distribution) on
$G$. The modified conclusion (b), see Theorem \ref{thm:p1}, is then: 

(b') $\mathscr{H}_{f}\xmapsto{\;T_{f}\;}L^{2}(\widehat{G},\mathscr{B},\mu)$,
$T_{f}\left(\varphi\ast f\right)=\widehat{\varphi}$, and 
\[
\left\Vert \varphi\ast f\right\Vert _{\mathscr{H}_{f}}^{2}=\int_{\widehat{G}}\left|\widehat{\varphi}\right|^{2}d\mu;
\]
where $\widehat{G}$ denotes the Pontryagin dual group to $G$, i.e.,
the group of all continuous characters $\chi$ on $G$; for functions
$\varphi$ on $G$ , the transform $\widehat{\varphi}$ is then 
\[
\widehat{\varphi}\left(\chi\right)=\int_{G}\chi\left(g\right)\varphi\left(g\right)dg
\]
with $dg$ denoting Haar measure on $G$; and further $\mu$ is a
Borel measure on $\widehat{G}$. 
\end{rem}
\begin{rem}
The main theme here is the interconnection between (i) the study of
extensions of locally defined continuous and positive definite  functions
$f$ on groups on the one hand, and, on the other, (ii) the question
of extensions for an associated system of unbounded Hermitian operators
with dense domain in a reproducing kernel Hilbert space (RKHS) $\mathscr{H}_{f}$
associated to $f$.

The analysis is non-trivial even if $G=\mathbb{R}^{n}$, and even
if $n=1$. If $G=\mathbb{R}^{n}$, we are concerned in (ii) with the
study of systems of $n$ skew-Hermitian operators $\left\{ S_{i}\right\} $
on a common dense domain in Hilbert space, and in deciding whether
it is possible to find a corresponding system of strongly commuting
selfadjoint operators $\left\{ T_{i}\right\} $ such that, for each
value of $i$, the operator $T_{i}$ extends $S_{i}$.

The version of this for non-commutative Lie groups $G$ will be stated
in the language of unitary representations of $G$, and corresponding
representations of the Lie algebra $La\left(G\right)$ by skew-Hermitian
unbounded operators.

In summary, for (i) we are concerned with partially defined positive
definite continuous functions $f$ on a Lie group; i.e., at the outset,
such a function $f$ will only be defined on a connected proper subset
in $G$. From this partially defined p.d. function $f$ we then build
a reproducing kernel Hilbert space $\mathscr{H}_{f}$, and the operator
extension problem (ii) is concerned with operators acting on $\mathscr{H}_{f}$,
as well as with unitary representations of $G$ acting on $\mathscr{H}_{f}$.
If the Lie group $G$ is not simply connected, this adds a complication,
and we are then making use of the associated simply connected covering
group. 

Because of the role of positive definite functions in harmonic analysis,
in statistics, and in physics, the connections in both directions
is of interest, i.e., from (i) to (ii), and vice versa. This means
that the notion of \textquotedblleft extension\textquotedblright{}
for question (ii) must be inclusive enough in order to encompass all
the extensions encountered in (i). For this reason enlargement of
the initial Hilbert space $\mathscr{H}_{f}$ are needed. In other
words, it is necessary to consider also operator extensions which
are realized in a dilation-Hilbert space; a new Hilbert space containing
$\mathscr{H}_{f}$ isometrically, and with the isometry intertwining
the respective operators. 

For more details, we refer the reader to \cite{MR3559001} and the
papers cited there.
\end{rem}
Key steps in the present application is the identification of a unitary
representation $\left\{ U_{t}\right\} _{t\in\mathbb{R}}$ acting in
the RKHS $\mathscr{H}_{f}$; it applies both when $f$ is continuous
and p.d. (Bochner), and when $f$ is a p.d. tempered distribution
(Schwartz \cite{MR0185423,MR0179587})
\[
U_{t}\left(\varphi\ast f\right):=\varphi\left(\cdot-t\right)\ast f=\varphi\ast f\left(\cdot+t\right).
\]
One concludes that
\begin{enumerate}
\item $U_{t_{1}}U_{t_{2}}=U_{t_{1}+t_{2}}$, $\forall t_{i}\in\mathbb{R}$;
and 
\item $U_{t}$ is strongly continuous, i.e., $\lim_{t\rightarrow0}\left\Vert U_{t}w-w\right\Vert _{\mathscr{H}_{f}}=0$
holds for $\forall w\in\mathscr{H}_{f}$. 
\end{enumerate}
Now assume that $f$ is p.d., and let $\left\{ U_{t}\right\} _{t\in\mathbb{R}}$
be the unitary group acting in the corresponding RKHS. By Stone's
theorem \cite{MR1503079}, there exists a projection valued measure
$E$ on $\left(\mathbb{R},\mathscr{B}\right)$, $\mathscr{B}=$ Borel
sigma-algebra in $\mathbb{R}$. That is, $E:\mathscr{B}\rightarrow\text{proj}(\mathscr{H}_{f})$
satisfying 
\begin{alignat}{2}
E\left(B\right)^{*} & =E\left(B\right)=E\left(B\right)^{2}, & \quad & \forall B\in\mathscr{B}\label{eq:a1}\\
E\left(B\cap C\right) & =E\left(B\right)=E\left(C\right), &  & \forall B,C\in\mathscr{B}\label{eq:a2}
\end{alignat}
such that 
\begin{equation}
U_{t}=\int_{\mathbb{R}}e^{i\lambda t}E\left(d\lambda\right),\;t\in\mathbb{R}.\label{eq:a3}
\end{equation}
(See, e.g., \cite{MR1503079}.)
\begin{lem}
Let $f$, $\mathscr{H}_{f}$, $\left\{ U_{t}\right\} _{t\in\mathbb{R}}$
be as above. Let $w_{0}=f\left(\cdot-0\right)\in\mathscr{H}_{f}$,
and set 
\begin{equation}
\mu:=\left\Vert E\left(d\lambda\right)w_{0}\right\Vert _{\mathscr{H}_{f}}^{2}\label{eq:a4}
\end{equation}
then 
\begin{equation}
\left\Vert \varphi\ast f\right\Vert _{\mathscr{H}_{f}}^{2}=\int_{\mathbb{R}}\left|\widehat{\varphi}\left(\lambda\right)\right|^{2}d\mu\left(\lambda\right).\label{eq:a5}
\end{equation}
\end{lem}
\begin{proof}
(\ref{eq:a4})$\Rightarrow$(\ref{eq:a5}) For $\varphi\in C_{c}\left(\mathbb{R}\right)$,
we have 
\[
\varphi\ast f:=\int_{\mathbb{R}}\varphi\left(t\right)f\left(\cdot-t\right)dt=\int\varphi\left(t\right)U_{t}w_{0}\,dt\;\left(\text{by \ensuremath{\left(\ref{eq:a1}\right)}, \ensuremath{\left(\ref{eq:a4}\right)}}\right)
\]
i.e., convolution, and 
\begin{eqnarray*}
\varphi\ast f & = & \int_{\mathbb{R}}\varphi\left(t\right)U_{t}w_{0}dt\\
 & \underset{\text{by \ensuremath{\left(\ref{eq:a3}\right)}}}{=} & \int_{\mathbb{R}}\varphi\left(t\right)\int e^{it\lambda}\underset{\text{PVM}}{\underbrace{E\left(d\lambda\right)}}w_{0}dt\\
 & \underset{\left(\text{Fubini}\right)}{=} & \int_{\mathbb{R}}\left(\int_{\mathbb{R}}\varphi\left(t\right)e^{it\lambda}dt\right)E\left(d\lambda\right)w_{0}\\
 &  & \int_{\mathbb{R}}\widehat{\varphi}\left(\lambda\right)E\left(d\lambda\right)w_{0}.
\end{eqnarray*}
It follows that 
\begin{eqnarray*}
\left\Vert \varphi\ast f\right\Vert _{\mathscr{H}_{f}}^{2} & \underset{\text{by \ensuremath{\left(\ref{eq:a2}\right)}}}{=} & \int_{\mathbb{R}}\left|\widehat{\varphi}\left(\lambda\right)\right|^{2}\left\Vert E\left(d\lambda\right)w_{0}\right\Vert _{\mathscr{H}_{f}}^{2}\\
 & \underset{\text{by \ensuremath{\left(\ref{eq:a4}\right)}}}{=} & \int_{\mathbb{R}}\left|\widehat{\varphi}\left(\lambda\right)\right|^{2}d\mu\left(\lambda\right),
\end{eqnarray*}
i.e., use $\left\langle E\left(d\lambda_{1}\right)w_{0},E\left(d\lambda_{2}\right)w_{0}\right\rangle _{\mathscr{H}_{f}}=\left\langle w_{0},E\left(d\lambda_{1}d\lambda_{2}\right)w_{0}\right\rangle _{\mathscr{H}_{f}}=\left\Vert E\left(d\lambda\right)w_{0}\right\Vert _{\mathscr{H}_{f}}^{2}$. 
\end{proof}
\textbf{Conclusion.} The RKHS $\mathscr{H}_{f}$ is the completion
of $\varphi\ast f$, $\varphi\in C_{c}\left(\mathbb{R}\right)$, with
respect to the norm $\int\left|\widehat{\varphi}\left(\lambda\right)\right|^{2}d\mu\left(\lambda\right)$,
via the mapping $\varphi\ast f\longmapsto\widehat{\varphi}\in L^{2}\left(\mathbb{R},\mu\right)$. 
\begin{rem}
Note that the proof is \emph{mutatis mutandis} to the case of positive
definite tempered distributions in the sense of L. Schwartz \cite{MR0185423,MR0179587}. 

We also get an extension of $\varphi\ast f$ to $\varphi\ast h$,
$\forall h\in\mathscr{H}_{f}\left(=\text{RKHS}\right)$ as follows. 

First define $\left\{ U_{t}\right\} _{t\in\mathbb{R}}$ as a unitary
representation of $\left(\mathbb{R},+\right)$ on $\mathscr{H}_{f}$;
and then, for $h\in\mathscr{H}_{f}$, set ($\varphi\in C_{c}\left(\mathbb{R}\right)$,
or $\varphi\in\mathcal{S}$)
\[
\varphi\ast h=\int_{\mathbb{R}}\varphi\left(t\right)U_{t}h\:dt.
\]
Then we have:
\[
\left\Vert \varphi\ast h\right\Vert _{\mathscr{H}_{f}}\leq\left(\int_{\mathbb{R}}\left|\varphi\left(t\right)\right|dt\right)\left\Vert h\right\Vert _{\mathscr{H}_{f}}=\left\Vert \varphi\right\Vert _{L^{1}}\left\Vert h\right\Vert ,\;\forall\varphi\in\mathcal{S},\:\forall h\in\mathscr{H}_{f}.
\]
\end{rem}
\begin{cor}
For every tempered positive definite measure $\mu$ (see Theorem \ref{thm:p1})
there is a unique Gaussian process $X=X^{\left(\mu\right)}$ indexed
by $x\in\mathbb{R}$, with mean zero, and variance 
\[
r^{\left(\mu\right)}\left(x\right)=\mathbb{E}\left(\big|X_{x}^{\left(\mu\right)}\big|^{2}\right)=\int_{\mathbb{R}}\left|1-e^{i\lambda x}\right|^{2}\frac{d\mu\left(\lambda\right)}{\lambda^{2}},
\]
and in addition, 
\[
\mathbb{E}\left(X_{x}^{\left(\mu\right)}\overline{X_{y}^{\left(\mu\right)}}\right)=\frac{r^{\left(\mu\right)}\left(\left|x\right|\right)+r^{\left(\mu\right)}\left(\left|y\right|\right)-r^{\left(\mu\right)}\left(\left|x-y\right|\right)}{2},
\]
and 
\[
\mathbb{E}\left(\big|X_{x}^{\left(\mu\right)}-X_{y}^{\left(\mu\right)}\big|^{2}\right)=r^{\left(\mu\right)}\left(\left|x-y\right|\right).
\]
\end{cor}
\begin{proof}
This family of stationary increment Gaussian processes were studied
in \cite{MR2793121}, and so we omit details here. The idea is to
apply the transform $T_{\mu}$ from Theorem \ref{thm:p1} (b) to the
associated Gaussian process. 

Setting $\varphi_{x}=\varphi=\begin{cases}
\chi_{\left[0,x\right]}\left(\cdot\right) & \text{if }x\geq0\\
-\chi_{\left[0,x\right]}\left(\cdot\right) & \text{if }x<0
\end{cases}$ , we get 
\begin{align*}
r^{\left(\mu\right)}\left(x\right) & =\int_{\mathbb{R}}\left|\widehat{\varphi}\left(\lambda\right)\right|^{2}d\mu\left(\lambda\right)\quad(\text{see Thm. \ref{thm:p1} (b)})\\
 & =\int_{\mathbb{R}}\left|1-e^{i\lambda x}\right|^{2}\frac{d\mu\left(\lambda\right)}{\lambda^{2}},\;x\in\mathbb{R},
\end{align*}
as claimed. 
\end{proof}

\section{Dirac Combs, and related Examples}

In Theorem \ref{thm:p1}, we made a distinction between the two cases:
that of (i) continuous p.d. functions, and (ii) the case of positive
definite tempered distributions. The two cases are connected with
the studies of Aronszajn \cite{MR0051437}, in case (i); and of Schwartz
\cite{MR0179587}, in case (ii). In the present section, we illustrate
this distinction in detail. 

To review the conclusions in Theorem \ref{thm:p1}, in case (i), the
Hilbert completion $\mathscr{H}_{f}$ of $\left\{ \varphi\ast f\mathrel{;}\varphi\in C_{c}\left(\mathbb{R}\right)\right\} $
in the pre-Hilbert inner product 
\begin{equation}
\int_{\mathbb{R}}\int_{\mathbb{R}}\varphi\left(x\right)\overline{\varphi\left(y\right)}f\left(x-y\right)dxdy
\end{equation}
is a \emph{reproducing kernel Hilbert space} (RKHS) in the sense of
Aronszajn. The reason is that, for all $x\in\mathbb{R}$, the function
$f\left(\cdot-x\right)$ is in $\mathscr{H}_{f}$, and 
\begin{equation}
\left\langle f\left(\cdot-x\right),h\right\rangle _{\mathscr{H}_{f}}=h\left(x\right),\;\forall x\in\mathbb{R},\:\forall h\in\mathscr{H}_{f};
\end{equation}
i.e., $\mathscr{H}_{f}$ satisfies the Aronszajn reproducing property.
This is not so in case (ii), but nonetheless, we shall get a reproducing
property in the \emph{measure theoretic} setting from the paper \cite{MR0179587}
of L. Schwartz.

The above conclusions are made precise in the following: 
\begin{prop}[The Dirac comb \cite{MR3456185,MR3494188,MR3004461}]
\label{prop:Dirac}Set 
\begin{equation}
\mu:=\sum_{n\in\mathbb{Z}}\delta_{n}\label{eq:d3}
\end{equation}
where $\delta_{n}$ in (\ref{eq:d3}) denotes the Dirac distribution.
Then $f=\widehat{\mu}$ is the tempered Schwartz distribution, written
formally as 
\begin{equation}
f\left(x\right)=\sum_{n\in\mathbb{Z}}e^{inx},\;x\in\mathbb{R}.\label{eq:d4}
\end{equation}
In this case the Hilbert completion $\mathscr{H}_{f}$ from Theorem
\ref{thm:p1} is the Hilbert space of all $2\pi$-periodic functions
$h$ on $\mathbb{R}$ subject to the condition 
\begin{equation}
\left\Vert h\right\Vert _{\mathscr{H}_{f}}^{2}:=\frac{1}{2\pi}\int_{-\pi}^{\pi}\left|h\left(x\right)\right|^{2}dx<\infty.\label{eq:d5}
\end{equation}
\end{prop}
\begin{proof}
A positive measure $\mu$ on $\mathbb{R}$ is said to be \emph{tempered}
iff $\exists M\in\mathbb{N}$ such that 
\begin{equation}
\int_{\mathbb{R}}\frac{d\mu\left(\lambda\right)}{1+\lambda^{2M}}<\infty.\label{eq:d6}
\end{equation}
The measure $\mu$ in (\ref{eq:d3}) is clearly tempered, and in particular
it is $\sigma$-finite. Specifically if $B\in\mathscr{B}_{\mathbb{R}}$
(the Borel $\sigma$-algebra), then 
\begin{equation}
\mu\left(B\right)=\#\left(B\cap\mathbb{Z}\right).
\end{equation}
For $M$ in (\ref{eq:d6}) we may take $M=1$. 

We now turn to the Hilbert completion $\mathscr{H}_{f}$ where $f$
is as in (\ref{eq:d4}). For all test-function $\varphi\in\mathcal{S}$,
we have: 
\begin{equation}
\left(\varphi\ast f\right)\left(x\right)=\sum_{n\in\mathbb{Z}}\widehat{\varphi}\left(n\right)e^{-inx}\label{eq:d8}
\end{equation}
where the interpretation of (\ref{eq:d8}) is in the sense of tempered
Schwartz distributions. Moreover, 
\begin{equation}
\int_{\mathbb{R}}\int_{\mathbb{R}}\varphi\left(x\right)\overline{\varphi\left(y\right)}f\left(x-y\right)dxdy=\sum_{n\in\mathbb{Z}}\left|\widehat{\varphi}\left(n\right)\right|^{2}.\label{eq:d9}
\end{equation}

Now, combining (\ref{eq:d8}) and (\ref{eq:d9}), we get that $\mathscr{H}_{f}$
is the Hilbert space described before (\ref{eq:d5}). To see this,
we apply the Plancherel-Fourier theorem, i.e., for $\forall\left(c_{n}\right)\in l^{2}$,
the function $h\left(x\right)=\sum_{n\in\mathbb{Z}}c_{n}e^{inx}$
is well defined, and 
\begin{equation}
\frac{1}{2\pi}\int_{-\pi}^{\pi}\left|h\left(x\right)\right|^{2}dx=\sum_{n\in\mathbb{Z}}\left|c_{n}\right|^{2}.
\end{equation}
Comparing now with (\ref{eq:d8}), the desired conclusion follows.
\end{proof}
\begin{rem}
By the Poisson summation formula, (\ref{eq:d4}) can also be written
as 
\[
f\left(x\right)=\sum_{n\in\mathbb{Z}}e^{inx}=2\pi\sum_{n\in\mathbb{Z}}\delta\left(x-2\pi n\right).
\]
\end{rem}

\subsection{The case of IFS-Cantor measures}

Let $\nu=\nu_{4}$ be the scale 4-Cantor fractal measure (see \cite{MR1215311,MR1655831})
specified by the IFS-identity: 
\begin{equation}
\frac{1}{2}\int\left(h\left(\frac{x}{4}\right)+h\left(\frac{x+2}{4}\right)\right)d\nu_{4}\left(x\right)=\int h\left(x\right)d\nu_{4}\left(x\right)\label{eq:d11}
\end{equation}
for all $h$. Introduce the transform 
\begin{equation}
\widehat{\nu}\left(\xi\right):=\int_{\mathbb{R}}e^{i\xi x}d\nu\left(x\right),
\end{equation}
and (\ref{eq:d11}) is equivalent to 
\begin{equation}
\widehat{\nu}_{4}\left(\xi\right)=\frac{1+e^{i\xi/2}}{2}\widehat{\nu}_{4}\left(\xi/4\right),\;\forall\xi\in\mathbb{R}.
\end{equation}
Note that, as a consequence, the support of this cantor measure $\nu_{4}$
is then precisely the scale-4 Cantor set from Fig \ref{fig:cantor}
above. It was shown by Jorgensen-Pedersen \cite{MR1655831} that $L^{2}\left(\nu_{4}\right)$
has an orthonormal basis (ONB) of functions $e_{\lambda}\left(x\right):=e^{i\lambda x}$.
One may take for example 
\begin{align}
\Lambda_{4} & :=\left\{ 0,1,4,5,16,17,20,21,64,65,\cdots\right\} \nonumber \\
 & =\left\{ \sum\nolimits _{0}^{\text{finite}}b_{j}4^{j}\mathrel{;}b_{j}\in\left\{ 0,1\right\} \right\} .\label{eq:d14}
\end{align}

\begin{figure}
\includegraphics[width=0.4\textwidth]{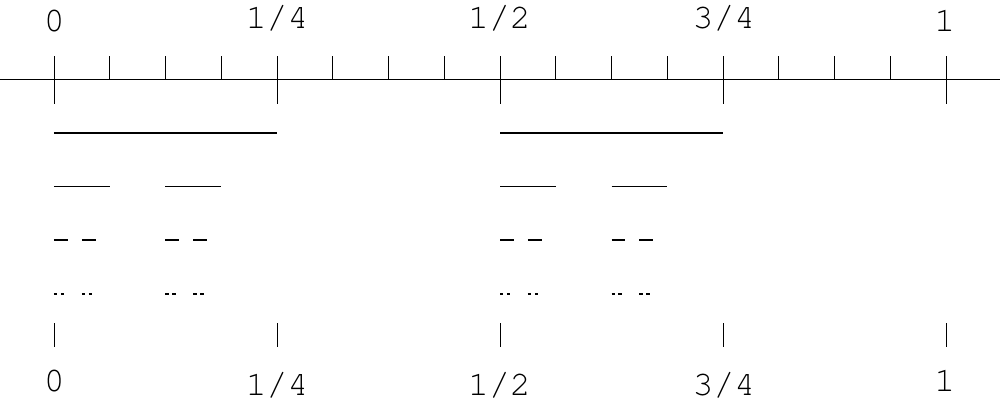}\caption{\label{fig:cantor}The $\frac{1}{4}$-Cantor set.}

\end{figure}

While $\left\{ e_{\lambda}\mathrel{;}\lambda\in\Lambda_{4}\right\} $
forms an ONB in $L^{2}\left(\nu_{4}\right)$, we say that $\left(\nu_{4},\Lambda_{4}\right)$
is a \emph{spectral pair}, it should be stressed that many Cantor
measures $\nu$ do \emph{not} allow ONBs of the form $\left\{ e_{\lambda}\mathrel{;}\lambda\in\Lambda\right\} $
for \emph{any} subsets $\Lambda$ of $\mathbb{R}$; for example $\nu_{3}$
is the opposite extreme: Jorgensen \& Pedersen proved that $L^{2}\left(\nu_{3}\right)$
does not admit more than two orthogonal functions of the form $e_{\lambda}\left(x\right)=e^{i\lambda x}$,
$\lambda\in\mathbb{R}$. By $\nu_{3}$, we mean the unique Borel probability
measure satisfying 
\begin{equation}
\frac{1}{2}\int\left(h\left(\frac{x}{3}\right)+h\left(\frac{x+2}{3}\right)\right)d\nu_{3}\left(x\right)=\int h\left(x\right)d\nu_{3}\left(x\right),
\end{equation}
for all $h$, compare (\ref{eq:d11}) with above. 

Using now the same ideas from the present paper, we get the following: 
\begin{prop}
Let $\left(\nu_{4},\Lambda_{4}\right)$ be as above; see (\ref{eq:d11})-(\ref{eq:d14}),
and set 
\[
\mu_{4}:=\sum_{\lambda\in\Lambda_{4}}\delta_{\lambda},
\]
and 
\[
f_{4}\left(x\right):=\sum_{\lambda\in\Lambda_{4}}e^{i\lambda x},\;x\in\mathbb{R},
\]
realized as a tempered p.d. distribution. Let $\mathscr{H}_{f_{4}}$
be the associated Hilbert space from Theorem \ref{thm:p1}. Then there
is a natural isometric isomorphism between the two Hilbert spaces
$\mathscr{H}_{f_{4}}$ and $L^{2}\left(\nu_{4}\right)$. 
\end{prop}
\begin{proof}
The details are the same as those of the proof of Proposition \ref{prop:Dirac}.
The key step is use of the fact from \cite{MR1655831} that $\left\{ e_{\lambda}\mathrel{;}\lambda\in\Lambda_{4}\right\} $
is an ONB in the Hilbert space $L^{2}\left(\nu_{4}\right)$ defined
from the Cantor measure $\nu_{4}$.
\end{proof}

\section{\label{sec:cor}Correspondences and Applications}
\begin{flushleft}
\par\end{flushleft}

\noindent\begin{minipage}[t]{1\columnwidth}%
\begin{minipage}[t]{0.48\columnwidth}%
\textbf{Continuous p.d. functions on $\mathbb{R}$}
\begin{lem*}
Let $f$ be a continuous function on $\mathbb{R}$. Then the following
are equivalent:
\begin{enumerate}[leftmargin=18pt]
\item $f$ is p.d., i.e., $\forall\varphi\in C_{c}\left(\mathbb{R}\right)$,
we have 
\begin{equation}
\int_{\mathbb{R}}\int_{\mathbb{R}}\varphi\left(x\right)\overline{\varphi\left(y\right)}f\left(x-y\right)dxdy\geq0.
\end{equation}
\item $\forall\left\{ x_{j}\right\} _{j=1}^{n}\subset\mathbb{R}$, $\forall\left\{ c_{j}\right\} _{j=1}^{n}\subset\mathbb{C}$,
and $\forall n\in\mathbb{N}$, we have 
\begin{equation}
\sum_{j=1}^{n}\sum_{k=1}^{n}c_{j}\overline{c}_{k}f\left(x_{j}-x_{k}\right)\geq0.
\end{equation}
 
\end{enumerate}
\end{lem*}
\end{minipage}\hfill{}%
\begin{minipage}[t]{0.48\columnwidth}%
\textbf{p.d. tempered distributions on $\mathbb{R}$}
\begin{lem*}
Let $f$ be a tempered distribution on $\mathbb{R}$. Then $f$ is
p.d. if and only if 
\begin{equation}
\int_{\mathbb{R}}\int_{\mathbb{R}}\varphi\left(x\right)\overline{\varphi\left(y\right)}f\left(x-y\right)dxdy\geq0
\end{equation}
hold, for all $\varphi\in\mathcal{S}$, where $\mathcal{S}$ is the
\textup{Schwartz space}. 

Equivalently, 
\begin{equation}
\left\langle f\left(x-y\right),\varphi\otimes\overline{\varphi}\right\rangle \geq0,\;\forall\varphi\in\mathcal{S}.
\end{equation}
Here $\left\langle \cdot,\cdot\right\rangle $ denotes distribution
action. 
\end{lem*}
\end{minipage}%
\end{minipage}
\begin{center}
\textbf{}%
\noindent\begin{minipage}[t]{1\columnwidth}%
\begin{center}
\textbf{RKHS}
\par\end{center}
\begin{minipage}[t]{0.48\columnwidth}%
\textbf{Bochner's theorem.} 

$\exists!$ positive finite measure $\mu$ on $\mathbb{R}$ such that
\[
f\left(x\right)=\int_{\mathbb{R}}e^{ix\lambda}d\mu\left(\lambda\right).
\]
\end{minipage}\hfill{}%
\begin{minipage}[t]{0.48\columnwidth}%
\textbf{Bochner/Schwartz}

$\exists$ positive tempered measure $\mu$ on $\mathbb{R}$ such
that 
\[
f=\widehat{\mu}
\]
where $\widehat{\mu}$ is in the sense of distribution. %
\end{minipage}%
\end{minipage}
\par\end{center}

\begin{center}
\textbf{}%
\noindent\begin{minipage}[t]{1\columnwidth}%
\begin{minipage}[t]{0.48\columnwidth}%
Let $\mathscr{H}_{f}$ be the RKHS of $f$. 
\begin{itemize}[leftmargin=10pt]
\item Then 
\begin{equation}
\left\Vert \varphi\ast f\right\Vert _{\mathscr{H}_{f}}^{2}=\int_{\mathbb{R}}\left|\widehat{\varphi}\left(\lambda\right)\right|^{2}d\mu\left(\lambda\right)
\end{equation}
where $\widehat{\varphi}=$ the Fourier transform. 
\item $f$ admits the factorization 
\[
f\left(x_{1}-x_{2}\right)=\left\langle f\left(\cdot-x_{1}\right),f\left(\cdot-x_{2}\right)\right\rangle _{\mathscr{H}_{f}}
\]
$\forall x_{1},x_{2}\in\mathbb{R}$, with 

$\mathbb{R}\ni x\longrightarrow f\left(\cdot-x\right)\in\mathscr{H}_{f}$. 
\end{itemize}
\end{minipage}\hfill{}%
\begin{minipage}[t]{0.48\columnwidth}%
Let $\mathscr{H}_{f}$ denote the corresponding RKHS. 
\begin{itemize}[leftmargin=10pt]
\item For all $\varphi\in\mathcal{S}$, we have 
\begin{equation}
\left\Vert \varphi\ast f\right\Vert _{\mathscr{H}_{f}}^{2}=\left\langle f\left(x-y\right),\varphi\otimes\overline{\varphi}\right\rangle ,
\end{equation}
distribution action. 
\item $\mathcal{S}\ni\varphi\longmapsto\varphi\ast f\in\mathscr{H}_{f}$,
where 
\[
\left(\varphi\ast f\right)\left(\cdot\right)=\int\varphi\left(y\right)f\left(\cdot-y\right)dy.
\]
\end{itemize}
\end{minipage}%
\end{minipage}
\par\end{center}

\begin{center}
\textbf{}%
\noindent\begin{minipage}[t]{1\columnwidth}%
\begin{center}
\textbf{Applications}
\par\end{center}
\begin{minipage}[t]{0.48\columnwidth}%
Now applied to Bochner's theorem. 

Set $\mathscr{H}_{f}=$ RKHS of $f$, and $w_{0}=f\left(\cdot-0\right)$.
Then 
\[
U_{t}w_{0}=w_{t}=f\left(\cdot-t\right),\;t\in\mathbb{R}
\]
defines a strongly continuous unitary representation of $\mathbb{R}$. %
\end{minipage}\hfill{}%
\begin{minipage}[t]{0.48\columnwidth}%
On white noise space:
\[
\mathbb{E}\left(e^{i\left\langle \varphi,\cdot\right\rangle }\right)=e^{-\frac{1}{2}\int\left|\widehat{\varphi}\right|^{2}d\mu}
\]
where $\mathbb{E}\left(\cdots\right)=$ expectation w.r.t the Gaussian
path-space measure. 

(The proof for the special case when $f$ is assumed p.d. and continuous
carries over with some changes to the case when $f$ is a p.d. tempered
distribution.)%
\end{minipage}%
\end{minipage}
\par\end{center}
\begin{note*}
In both cases, we have the following representation for vectors in
the RKHS $\mathscr{H}_{f}$: 
\begin{equation}
\left\langle \varphi\ast f,\psi\ast f\right\rangle _{\mathscr{H}_{f}}=\left\langle \varphi\ast\overline{\psi},f\right\rangle ,\;\forall\varphi,\psi\in\mathcal{S};
\end{equation}
where  $\varphi\ast f:=$ the standard convolution w.r.t. Lebesgue
measure.
\end{note*}

\section{Unimodular groups}

Let $G$ be a locally compact group, and assume it is unimodular,
i.e., its Haar measure is both left and right invariant. By a theorem
of I.E. Segal, there is then a Plancherel theorem for the unitary
representations of $G$ (see \cite{MR0036765} and \cite{MR1043174,MR0396826,MR1187300}).
If $C^{*}\left(G\right)$ denotes the group algebra with convolution
product 
\begin{equation}
\left(\varphi\ast\psi\right)\left(x\right)=\int_{G}\varphi\left(y\right)\psi\left(y^{-1}x\right)dy,\label{eq:u1}
\end{equation}
where $\varphi$, $\psi$ are functions on $G$, and $dy$ denotes
the Haar measure. The $\ast$-operation in (\ref{eq:u1}) is 
\begin{equation}
\varphi^{*}\left(x\right)=\overline{\varphi\left(x^{-1}\right)},\;x\in G.\label{eq:u2}
\end{equation}
Then $C^{*}\left(G\right)$ is the $C^{*}$-completion of this $*$-algebra.

Note that since $G$ is assumed unimodular, we need not include the
modular function $\Delta$ in the definition (\ref{eq:u2}). By general
theory, it is known that the set of equivalence classes of irreducible
unitary representations of $G$ is then in bijective correspondence
with the set $P\left(G\right)$ of pure states of $C^{*}\left(G\right)$. 
\begin{lem}
(a) Let $G$ be a unimodular (locally compact) group, and let $f$
be a continuous positive definite function on $G$. Let $\mathscr{H}_{f}$
be the corresponding reproducing kernel Hilbert space (RKHS) If $\pi$
is an irreducible unitary representation of $G$, we denote by $\lambda_{\pi}$
the corresponding state. More precisely, 
\begin{equation}
\lambda_{\pi}\left(x\right)=\left\langle v,\pi\left(x\right)v\right\rangle _{\mathscr{H}_{\pi}},\;x\in G\label{eq:u3}
\end{equation}
defines a pure state, $\lambda_{\pi}\in P\left(G\right)$. 

(b) Given $f$ p.d. and continuous as above, there is a unique Borel
measure $\mu=\mu_{f}$ concentrated on $P\left(G\right)$ such that
\begin{equation}
\left\Vert \varphi\ast f\right\Vert _{\mathscr{H}_{f}}^{2}=\int_{P\left(G\right)}\left|\lambda_{\pi}\left(\varphi\right)\right|^{2}d\mu\left(\lambda_{\pi}\right).\label{eq:u4}
\end{equation}
\end{lem}
\begin{proof}
First a caution, the set $P\left(G\right)$ may not in general be
a Borel set, but by a theorem of Phelphs \cite{MR0435818}, the measure
$\mu$ may be chosen on a Borel set $B$ such that 
\begin{equation}
\mu\left(B\,\Delta\,P\left(G\right)\right)=0\label{eq:u5}
\end{equation}
where $\Delta$ denotes ``symmetric difference.''

Other than this point, the present proof follows closely that of Section
\ref{sec:para} (in the Abelian case).

We introduce $\mathscr{H}_{f}$ as the completion of the functions
$\varphi\ast f$ (convolution) for $\varphi\in C_{c}\left(G\right)$:
\[
\left\Vert \varphi\ast f\right\Vert _{\mathscr{H}_{f}}^{2}=\int_{G}\left(\varphi\ast\varphi^{*}\right)\left(x\right)f\left(x\right)dx,
\]
see (\ref{eq:u1})-(\ref{eq:u1}); with $dx$ denotes Haar measure.
As before, we get a unitary representation $U$ of $G$ acting in
$\mathscr{H}_{f}$ via 
\[
U_{x}\left(\varphi\ast f\right)=\varphi\left(x^{-1}\cdot\right)\ast f,
\]
and $\left(\left\{ U_{x}\right\} _{x\in G},\mathscr{H}_{f}\right)$
then decomposes as per the Plancherel theorem for $G$. Hence there
exists a unique $\mu$ on $P\left(G\right)$ such that 
\[
U=\int_{P\left(G\right)}\pi\,d\mu\left(\lambda_{\pi}\right),
\]
and the result follows. 
\end{proof}
\begin{acknowledgement*}
The co-authors thank the following colleagues for helpful and enlightening
discussions: Professors Daniel Alpay, Sergii Bezuglyi, Ilwoo Cho,
A. Jaffe, Paul Muhly, K.-H. Neeb, G. Olafsson, Wayne Polyzou, Myung-Sin
Song, and members in the Math Physics seminar at The University of
Iowa. 
\end{acknowledgement*}
\bibliographystyle{amsalpha}
\bibliography{ref}

\end{document}